\documentclass[10pt, reqno]{amsart}

\newtheorem{theorem}{Theorem}[section]

\newtheorem{pro}[theorem]{Proposition}
\newtheorem{lem}[theorem]{Lemma}

\theoremstyle{definition}
\newtheorem{de}[theorem]{Definition}
\theoremstyle{remark}
\newtheorem{re}[theorem]{Remark}
\numberwithin{equation}{section}
\usepackage{amsmath}
\usepackage{amssymb}
\usepackage{amsthm}
\usepackage[utf8]{inputenc}
\usepackage[normalem]{ulem}
\usepackage{hyperref}
\usepackage{mathtools}
\mathtoolsset{showonlyrefs}
\renewcommand{\Re}{\operatorname{Re}}
\renewcommand{\Im}{\operatorname{Im}}
\usepackage[top=3cm, bottom=2cm, left=3cm, right=3cm]{geometry}
\usepackage[abbrev]{amsrefs}
\usepackage{bbm}
\renewcommand{\MR}[1]{}

\usepackage{slashed}
\begingroup 
\makeatletter
   \@for\theoremstyle:=theorem,de,re,pro,lem,coro,plain\do{%
     \expandafter\g@addto@macro\csname th@\theoremstyle\endcsname{%
        \addtolength\thm@preskip\parskip
     }%
   }
\endgroup

\begin{document}

\title[Scattering in the non-radial INLS]{A Virial-Morawetz approach to scattering for the non-radial inhomogeneous NLS}

\author[L. Campos]{Luccas Campos}
\address{Department of Mathematics, UFMG, Brazil}
\email{luccasccampos@gmail.com}
\thanks{}

\author[M. Cardoso]{Mykael Cardoso}
\address{Department of Mathematics, UFPI, Brazil}
\email{mykael@ufpi.edu.br}
\thanks{}

\makeatletter{\renewcommand*{\@makefnmark}{}
\makeatother}
\makeatletter{\renewcommand*{\@makefnmark}{}
\makeatother}
\begin{abstract}
Consider the focusing inhomogeneous nonlinear Schrödinger equation in $H^1(\mathbb{R}^N)$,
		\begin{equation}
			iu_t + \Delta u + |x|^{-b}|u|^{p-1}u=0,\\
		\end{equation}
		when $b > 0$ and $N \geq 3$ in the intercritical case $0 < s_c <1$. In previous works, the second author, as well as Farah, Guzmán and Murphy, applied the concentration-compactness approach to prove scattering below the mass-energy threshold for radial and non-radial data. Recently, the first author adapted the Dodson-Murphy approach for radial data, followed by Murphy, who proved scattering for non-radial solutions in the 3d cubic case, for $b<1/2$. This work generalizes the recent result of Murphy, allowing a broader range of values for the parameters $p$ and $b$, as well as allowing any dimension $N \geq 3$. It also gives a simpler proof for scattering nonradial, avoiding the Kenig-Merle road map. We exploit the decay of the nonlinearity, which, together with Virial-Morawetz-type estimates, allows us to drop the radial assumption.
\end{abstract}

\maketitle

\section{Introduction}

In this work, we consider the Cauchy problem for the focusing inhomogeneous nonlinear Schrö\-din\-ger equation (INLS)

\begin{equation}\label{INLS}
\begin{cases}
iu_t + \Delta u + |x|^{-b}|u|^{p-1}u=0,\\
			u(0) = u_0 \in H^1(\mathbb{R}^N), 
\end{cases}
\end{equation}

where $u: \mathbb{R}^N\times\mathbb{R} \to \mathbb{C}$, $N \geq 3$, $0 < b < 2$, and
\begin{equation}\label{cond_p}
	 1 + \frac{4-2b}{N} < p < 1+\frac{4-2b}{N-2}.
\end{equation}

These equations arise as a model in optics, to accounts for the inhomogeneity of the medium. For a physical point of view, we refer to Gill \cite{Gill}, Liu and Tripathi \cite{Liu}. The INLS case appears naturally as a limiting case of potentials that decay as $|x|^{-b}$ at infinity (Genoud and Stuart \cite{g_8}).


Moreover, this model is invariant under scaling. Indeed, if $u(x,t)$ is a solution to \eqref{INLS}, then
\begin{equation}
	u_\lambda(x,t) = \lambda^\frac{2-b}{p-1}u(\lambda x, \lambda^2 t), \quad \lambda > 0,
\end{equation}
is also a solution. Computing the homogeneous Sobolev norm, we obtain
\begin{equation}
	\|u_\lambda(\cdot,0)\|_{\dot{H}^s} = \lambda^{s-\frac{N}{2}-\frac{2-b}{p-1}}\|u_0\|_{\dot{H}^s}.
\end{equation}

The Sobolev index which leaves the scaling symmetry invariant is called the \textit{critical index} and is defined as
\begin{equation}
s_c = \frac{N}{2} - \frac{2-b}{p-1}.
\end{equation}

Note that the condition \eqref{cond_p} is equivalent to $0 < s_c < 1$.

Solutions to the Cauchy problem \eqref{INLS} conserve mass $M[u]$ and energy $E[u]$, defined by
\begin{equation}\label{def_mass}
M\left[u(t) \right] = \int |u(t)|^2 dx = M[u_0],
\end{equation}
\begin{equation}\label{def_energy}
E\left[u(t) \right] = \frac{1}{2}\int |\nabla u(t)|^2 dx - \frac{1}{p+1} \int |x|^{-b}|u(t)|^{p+1} dx = E[u_0].
\end{equation}

The homogeneous case $b = 0$ is known as the nonlinear Schr\"odinger (NLS) equation, which has been receiving attention over the past decades (see, for instance, the works of Bourgain \cite{Bo99}, Cazenave \cite{cazenave}, Linares-Ponce \cite{LiPo15} and Tao \cite{TaoBook}).

We briefly review the literature about \eqref{INLS}. Genoud and Stuart \cite{g_8} proved that \eqref{INLS} is locally well-posed in $H^1(\mathbb{R}^N)$, $N \geq 1$ for $0 < b < \min\{2,N\}$. For other well-posedness results for this equation, we refer the reader to Guzmán \cite{Boa} and  Dinh \cite{Boa_Dinh}. Farah \cite{Farah_well} proved global well-posedness for the INLS in $H^1(\mathbb{R}^N)$ if 

\begin{equation}\label{ME<1}
    M[u_0]^\frac{1-s_c}{s_c}E[u_0] < M[Q]^\frac{1-s_c}{s_c}E[Q]
\end{equation}
and 
\begin{equation}\label{MK<1}
    \|u_0\|_{L^2}^\frac{1-s_c}{s_c}\|\nabla u_0\|_{L^2} < \|Q\|_{L^2}^\frac{1-s_c}{s_c}\|\nabla Q\|_{L^2},
\end{equation}

where $Q$ is the unique positive radial solution to the elliptic equation 
\begin{equation}\label{def_Q}
	\Delta Q - Q + |x|^{-b}|Q|^{p-1}Q = 0,
\end{equation}

usually referred as the \textit{ground state} associated to \eqref{INLS}.

Scattering in $H^1$ under \eqref{ME<1} and \eqref{MK<1} was initially proved using the concentration-compactness-rigidity approach in the radial setting for for $N \geq 2$ by Farah-Guzmán \cites{FG_Scat_3d, FG_scat}, by imposing some extra restrictions on $p$ and $b$. The first author \cite{Campos_New_2019} generalized the results for the whole intercritical setting in $N \geq 3$, extending the allowed range for $p$ and $b$, by adapting the ideas of Dodson-Murphy \cite{MD_New} to the radial NLS. 

For the non-radial case, the lack of momentum conservation posed a technical difficulty. In the concentration-compactness-rigidity approach, a critical solution is constructed, whose orbit is compact under some symmetries, one of which is the translation parameter $x(t)$. In the homogeneous case ($b=0$), the translation parameter associated to a zero-momentum (critical) solution under \eqref{ME<1} and \eqref{MK<1} which does not scatter satisfies $x(t) = o(t)$. However, for the INLS, this control is not available through momentum arguments. Moreover, the non-radial interaction Morawetz approach by Dodson-Murphy \cite{MD_non_radial} fails, due to the same lack of conservation.

However, it is possible to make use of the spatial decay of the nonlinearity in the INLS equation to extend to the non-radial case the proofs used in the radial case. This was shown in \cite{CFGM} by an adapted profile decomposition which eventually concluded that one could take $x(t) \equiv 0$. As such, non-radial (critical) solutions to the INLS equation behaved similarly, in some sense, to the radial ones. Recently, Murphy \cite{Jason2021} proved that a similar intuition works for the Virial-Morawetz approach, at least in the $3d$ cubic case, when $0 < b < 1/2$. Here, inspired by \cite{Jason2021}, we show how to formalize this intuition in the non-radial case, for any $N \geq 3$, $0< s_c < 1$.

The key to main result is the scattering criterion, which was first proved for the $3d$ cubic NLS equation by Tao \cite{Tao_Scat} (see also \cites{Campos_New_2019, AndyScat,Jason2021}). 
\begin{theorem}[Scattering criterion]\label{scattering_criterion}
Let  $N \geq 3$, $1+\frac{4-2b}{N} < p < 1+\frac{4-2b}{N-2}$ and $0 < b < \min\{N/2,2\}$. Consider an $H^1(\mathbb{R}^N)$-solution $u$ to \eqref{INLS} defined on $[0,+\infty)$ and assume the a priori bound 
\begin{equation}\label{E}
\displaystyle\sup_{t \in [0,+\infty)}\left\|u(t)\right\|_{H^1_x} := E < +\infty.
\end{equation}

There exist constants $R > 0$ and $\epsilon>0$ depending only on $E$, $N$, $p$ and $b$ (but never on $u$ or $t$) such that if
\begin{equation}\label{scacri}
\liminf_{t \rightarrow +\infty}\int_{B(0,R)}|u(x,t)|^2 \, dx \leq \epsilon^2,
\end{equation}
then there exists a function $u_+ \in H^1(\mathbb{R}^N)$ such that 
$$
\lim_{t \rightarrow +\infty}\left\|u(t)-e^{it\Delta}u_+\right\|_{H^1(\mathbb{R}^N)} = 0,
$$ 
i.e., $u$ scatters forward in time in $H^1(\mathbb{R}^N)$.
\end{theorem}
\begin{re}
The criterion above was proved for radial solutions in \cite{Campos_New_2019}, and relied heavily on the so-called Strauss Lemma, which ensures spatial localization of radial $H^1$ functions. Here, we drop the radiality assumption, showing that the exact same criterion applies to non-radial solutions as well. This shows that the decay of the nonlinearity implies in some kind of localization for solutions under the thresholds given by the ground state.
\end{re}

The localization effect caused by the decay of the nonlinearity can be expressed as the following proposition, which we show to hold for non-radial solutions.
\begin{pro}[Virial-Morawetz estimate]\label{virial}
For $N \geq 3$, $1+\frac{4-2b}{N} < p < 1+\frac{4-2b}{N-2}$ and $0 < b < \min\{N/2,2\}$, let $u$ be a $H^1(\mathbb{R}^N)$-solution to \eqref{INLS} satisfying \eqref{ME<1} and \eqref{MK<1}. Then, there exists $R >0$ such that, for any $T>0$,
$$
\frac{1}{T}\int_0^T\int_{|x|\leq R}|x|^{-b}|u(x,t)|^{p+1}\,dx\, dt \lesssim_{u,\delta} \frac{R}{T}+\frac{1}{R^b}.
$$
\end{pro}

The scattering criterion and the virial-Morawetz estimates allow us to prove the following theorem.
\begin{theorem}\label{teo1} Let $N \geq 3$,  $1+\frac{4-2b}{N} < p < 1+\frac{4-2b}{N-2}$, $0 < b < \min\{N/2,2\}$, and $u_0 \in H^1(\mathbb{R}^N)$ be such that 
$$M[u_0]^\frac{1-s_c}{s_c}E[u_0] < M[Q]^\frac{1-s_c}{s_c}E[Q]$$ and $$\|u_0\|_{L^2}^\frac{1-s_c}{s_c}\|\nabla u_0\|_{L^2} < \|Q\|_{L^2}^\frac{1-s_c}{s_c}\|\nabla Q\|_{L^2}.$$ 
Then the solution $u(t)$ to \eqref{INLS} exists globally in time and scatters in $H^1$ in both time directions.
\end{theorem}

\begin{re}
The proofs in \cites{FG_Scat_3d, FG_scat, MMZ, CFGM} use the so-called concentration-com\-pact\-ness-rigidity approach, pionereed by Kenig and Merle \cite{KM_Glob} in the context of the energy-critical ($s_c = 1$) NLS equation. More recently, Dodson and Murphy \cite{MD_New} developed a new approach, based on Tao's scattering criterion in \cite{Tao_Scat} and on \textit{Virial-Morawetz estimates}. This approach was adapted to the INLS by \cite{Campos_New_2019}, in the radial case, and by Murphy \cite{Jason2021} in the $3d$ cubic, non-radial case.  We develop here a modification of the approach in \cite{Jason2021}, closer to the one chosen in \cite{Campos_New_2019}, replacing $L_t^4 W^{1,3}_x$ estimates by smoother Strichartz estimates which, together with small data theory, make it possible to handle the inhomogeneity better, allowing for an optimal range of parameters in dimensions $N\geq 3$. The radial assumption is droped vis-à-vis the $|x|^{-b}$ factor in the nonlinear term. In lower dimensions, this approach fails due to the slow decay on time of the Schr\"odinger operator $e^{it\Delta}$ and the slow decay in the Virial-Morawetz estimate due to the weaker non-radial decay.
\end{re}

This paper is organized as follows: in the next section, we introduce some notation and basic estimates. In Section $3$, we prove the scattering criterion (Theorem \ref{scattering_criterion}). In Section $4$, we apply this criterion, together with Morawetz/Virial estimates to prove Theorem \ref{teo1}.

\section{Notation and basic estimates}

We denote by $p'$ the Holder's conjugate of $p \geq 1$. We use $X \lesssim Y$  to denote $X \leq C Y$, where the constant $C$ only depends on the parameters (such as $N$, $p$, $b$, as well as $E$ in \eqref{E}) and exponents, but never on $u$ or on $t$. The notations $a^+$ and $a^-$ denote, respectively, $a+\eta$ and $a-\eta$, for a fixed $0 < \eta \ll 1$. We use $p^*$ to denote the critical exponent of the Sobolev embedding $H^1 \hookrightarrow L^{p^*}$, that is, $p^* = 2N/(N-2)$, for $N \geq 3$.

\begin{de}\label{Hs_adm} If $N \geq 1$ and $s \in (-1,1)$, the pair $(q,r)$ is called $\dot{H}^s$\textit{-admissible} if it satisfies the condition
\begin{equation}\label{hs_adm_eq}
\frac{2}{q} = \frac{N}{2}-\frac{N}{r}-s,    
\end{equation}
where
$$
2 \leq q,r \leq \infty, \text{ and } (q,r,N) \neq (2,\infty,2).
$$
In particular, if $s=0$, we say that the pair is $L^2$-admissible.
\end{de}

\begin{de}\label{As}
Given $N> 2$, consider the set
\begin{equation}
    \mathcal{A}_0 = \left\{(q,r)\text{ is } L^2\text{-admissible} \left|
    \, 2 \leq r \leq \frac{2N}{N-2}
    \right.
    \right\}.
\end{equation}
For $N>2$ and $s \in (0,1)$, consider also
\begin{equation}
    \mathcal{A}_s = \left\{(q,r)\text{ is } \dot{H}^{s}\text{-admissible} \left|\,
    \left(\frac{2N}{N-2s}\right)^+ \leq r \leq \left(\frac{2N}{N-2}\right)^-
    \right.
    \right\}
\end{equation}
and
\begin{equation}
    \mathcal{A}_{-s} = \left\{(q,r) \text{ is } \dot{H}^{-s}\text{-admissible} \left|
    \left(\frac{2N}{N-2s}\right)^+ \leq r \leq \left(\frac{2N}{N-2}\right)^-
    \right.\right\}.
\end{equation}
We define the following Strichartz norm
\begin{equation}
	\|u\|_{S(\dot{H}^s,I)} = \sup_{(q,r)\in \mathcal{A}_s}\|u\|_{L_I^qL_x^r},	
\end{equation}

and the dual Strichartz norm
\begin{equation}
	\|u\|_{S'(\dot{H}^{-s},I)} = \inf_{(q,r)\in \mathcal{A}_{-s}}\|u\|_{L_I^{q'}L_x^{r'}}.	
\end{equation}
If $s=0$, we shall write $S(\dot{H}^0,I) = S(L^2,I)$ and $S'(\dot{H}^0,I) = S'(L^2,I)$. If $I=\mathbb{R}$, we will omit $I$.
\end{de}


\subsection{Strichartz Estimates }

In this work, we use the following versions of the Strichartz estimates:

\begin{itemize}
\item[(i)] \textbf{The standard Strichartz estimates} (Cazenave \cite{cazenave}, Keel and Tao \cite{KT98}, Foschi \cite{Foschi05})

\begin{equation}\label{S1}
\|e^{it\Delta}f\|_{S(L^2)} \lesssim\|f\|_{L^2},
\end{equation}

\begin{equation}\label{S2}
\|e^{it\Delta}f\|_{S(\dot{H}^s)} \lesssim\|f\|_{\dot{H}^s},
\end{equation}
\noeqref{S2}
\begin{equation}\label{KS1}
\left\|\int_\mathbb{R}e^{i(t-\tau)\Delta}g(\cdot,\tau) \, d\tau\right\|_{S(L^2,I)} + \left\|\int_0^te^{i(t-\tau)\Delta}g(\cdot,\tau) \, d\tau\right\|_{S(L^2,I)}\lesssim\|g\|_{S'(L^2,I)}.
\end{equation}
\noeqref{KS1}
\item[(ii)] \textbf{The Kato-Strichartz estimate} (Kato \cite{Kato94}, Foschi \cite{Foschi05})
\begin{equation}\label{KS2}
\left\|\int_\mathbb{R}e^{i(t-\tau)\Delta}g(\cdot,\tau) \, d\tau\right\|_{S(\dot{H}^s,I)} + \left\|\int_0^te^{i(t-\tau)\Delta}g(\cdot,\tau) \, d\tau\right\|_{S(\dot{H}^s,I)}\lesssim\|g\|_{S'(\dot{H}^{-s},I)}.
\end{equation}
\item[(iii)] \textbf{Local-in-time estimate}

\begin{equation}\label{lit}
\left\|\int_a^be^{i(t-\tau)\Delta}g(\cdot,\tau) \, d\tau\right\|_{S(\dot{H}^{s},\mathbb{R})}\lesssim\|g\|_{S(\dot{H}^{-s}, [a,b])}.
\end{equation}
\end{itemize}
These relations are obtained from the decay of the linear operator (see, for instance, Linares and Ponce \cite[Lemma 4.1]{LiPo15})
\begin{equation}\label{decay}
    \|e^{it\Delta}f\|_{L^p_x} \lesssim \frac{1}{|t|^{\frac{N}{2} \left(\frac{1}{p'} - \frac{1}{p}\right)}} \|f\|_{L^{p'}_x}, \quad p \geq 2, 
\end{equation}
combined with Sobolev inequalities and interpolation. The inequalities \eqref{S1}-\eqref{KS2} are standard in the theory \cite{cazenave}. The inequality \eqref{lit} follows from \eqref{KS2} by noting that
\begin{equation}
    \int_a^be^{i(t-\tau)\Delta}g(\tau) \, d\tau = \int_{\mathbb{R}}e^{i(t-\tau)\Delta}\mathbbm{1}_{[a,b]}(\tau)g(\tau) \, d\tau.
\end{equation}

\subsection{Other useful estimates}




In what follows we also use the following standard estimates.

\begin{lem}
[{See \cite[Section 2]{Campos_New_2019} and \cite[Section 4]{Boa}}]
\label{lem_guz} 
Let $N \geq 3$, $u, v \in C^\infty_0(\mathbb{R}^{N+1})$, $1+\frac{4-2b}{N} < p < 1+\frac{4-2b}{N-2}$ and $0 < b <\min\{N/2,2\}$. Then there exists $0 \leq \theta = \theta(N,p,b) \ll p-1$ such that the following inequalities hold
\begin{align}
\label{dual_s}\left\||x|^{-b}|u|^{p-1}v\right\|_{S'(\dot{H}^{-s_c},I)}&\lesssim\left\|u\right\|^\theta_{L^\infty_tH^1_x}\left\|u\right\|^{p-1-\theta}_{S(\dot{H}^{s_c},I)}\left\|v\right\|_{S(\dot{H}^{s_c},I)},\\
\noeqref{dual_l}\label{dual_l}\left\||x|^{-b}|u|^{p-1}u\right\|_{S'\left(L^2, I\right)}&\lesssim\left\|u\right\|^\theta_{L^\infty_tH^1_x}\left\|u\right\|^{p-1-\theta}_{S\left(\dot{H}^{s_c}, I\right)}\left\|u\right\|_{S\left(L^2, I\right)},\\
\label{dual_grad_l}\left\|\nabla\left(|x|^{-b}|u|^{p-1}u\right)\right\|_{S'\left(L^2, I\right)}
&\lesssim\left\|u\right\|^\theta_{L^\infty_tH^1_x}\left\|u\right\|^{p-1-\theta}_{S
\left(\dot{H}^{s_c}, I\right)}\left\|\nabla u\right\|_{S\left(L^2, I\right)}^{},\\
\label{L1}\||x|^{-b}|u|^{p-1}u\|_{L^\infty_IL^{r}_x} &\lesssim \|u\|_{L^\infty_I H^1_x}^p\
\end{align}
for $\frac{2(N-b)}{N+4-2b} < r < \frac{2(N-b)}{N+2-2b}.$
\end{lem}

\begin{re}\label{re_guz}
Inequalities \eqref{dual_s}-\eqref{dual_grad_l} were proved in \cite{Boa} for $0 < b < b^*$ ($b^*=\frac{N}{3}$, if $N=1,2,3$ and $b^*=2$, if $N\geq 4$) and with the additional restriction $p < 4-2b$ instead of $p < 5-2b$ in the 3d case. The proof was extended to the full range in \cite{Campos_New_2019}.
\end{re}
The next lemma was proved in \cite{Boa} with the same restrictions mentioned in Remark \ref{re_guz}. In view of the results in \cite{Campos_New_2019}, the proof in \cite{Boa} immediately extended to the new range of $p$ and $b$.
\begin{lem}[{Small data theory, see \cite{Campos_New_2019} and \cite{Boa}}]\label{small_data} Let $N \geq 3$, $1+\frac{4-2b}{N} < p < 1+\frac{4-2b}{N-2}$ and $0 < b <\min\{N/2,2\}$. Suppose $\|u_0\|_{H^1} \leq E$. Then there exists $\delta_{sd} = \delta_{sd}(E) > 0$ such that if 
\begin{equation}
	\|e^{it\Delta}u_0\|_{S(\dot{H}^{s_c},[0,+\infty))} \leq \delta_{sd},
\end{equation}

then the solution $u$ to \eqref{INLS} with initial condition $u_0 \in H^1(\mathbb{R}^N)$ is globally defined on $[0,+\infty)$. Moreover,
\begin{equation}
		\|u\|_{S(\dot{H}^{s_c},[0,+\infty))} \leq 2 	\|e^{it\Delta}u_0\|_{S(\dot{H}^{s_c},[0,+\infty))},
\end{equation}
and
\begin{equation}
	\|u\|_{S(L^2, [0,+\infty))}+\|\nabla u\|_{S(L^2, [0,+\infty))} \lesssim \|u_0\|_{H^1}.
\end{equation}
Furthermore, $u$ scatters forward in time in $H^1$, i.e., here exists $u_+ \in H^1$ such that
\begin{equation}
	\lim_{t\to+\infty} \|u(t)-e^{it\Delta}u_+\|_{H^1_x} = 0.
\end{equation}
\end{lem}

\section{Proof of the scattering criterion}

The following result is the key to prove Theorem \ref{scattering_criterion}. It was proved initially for radial solutions to the INLS equation, in the intercritical setting, for $N \geq 3$, in \cite{Campos_New_2019}. Jason \cite{Jason2021} extended the result for non-radial data in the 3d-cubic setting, for $0 < b < 1/2$. Here, we prove the result for non-radial data the full intercritical range, for $N \geq 3$.
 \begin{lem}\label{linevo}
Let  $N \geq 3$, $1+\frac{4-2b}{N} < p < 1+\frac{4-2b}{N-2}$, $0 < b < \min\{N/2,2\}$ and $u$ be a (possibly non-radial) $H^1(\mathbb{R}^N)$-solution to \eqref{INLS} satisfying \eqref{E}. If $u$ satisfies \eqref{scacri} for some $0 < \epsilon < 1$, then there exist $\gamma, T > 0$ such that the following estimate is valid
\begin{equation}
\left\|e^{i(\cdot-T)\Delta}u(T)\right\|_{S\left(\dot{H}^{s_c}, [T, +\infty)\right)}  \lesssim\epsilon^\gamma.
\end{equation}
\end{lem}
 \begin{proof} 
 
 For $N \geq 3$, fix the parameters $\alpha, \gamma >0$ (to be chosen later). From \eqref{S2}, there exists $T_{0} > \epsilon^{-\alpha}$ such that
\begin{equation}\label{T0}
\left\|e^{it\Delta}u_0\right\|_{S\left(\dot{H}^{s_c}, [T_0,+\infty) \right)} \leq \epsilon^{\gamma}.
\end{equation}

For $T\geq T_0$ to be chosen later, define  $I_1 :=\left[T-\epsilon^{-\alpha}, T\right]$, $I_2 := [0, T-\epsilon^{-\alpha}]$  and let $\eta$ denote a smooth, spherically symmetric function which equals $1$ on $B(0, 1/2)$ and $0$ outside $B(0,1)$. For any $R > 0$ use $\eta_R$ to denote the rescaling $\eta_R(x) := \eta(x/R)$. 

From Duhamel's formula
\begin{equation}
    u(T) = e^{iT\Delta}u_0 + i\int_0^{T}e^{i(T-s)\Delta}|x|^{-b}|u|^{p-1}u(s) \, ds,
\end{equation}
we obtain
\begin{equation}
e^{i(t-T)\Delta}u(T)  = e^{it\Delta}u_0 + iF_1 + iF_2,
\end{equation}
where, for $i = 1,2,$
$$
F_i = \int_{I_i} e^{i(t-s)\Delta}|x|^{-b}|u|^{p-1}u(s) \,ds.
$$
We refer, as usual, to $F_1$ as the ``recent past", and to $F_2$ as the ``distant past". By \eqref{T0}, it remains to estimate $F_1$ and $F_2$.

\textbf{Step 1. Estimate on recent past.}

By hypothesis \eqref{scacri}, we can fix $T\geq T_0$ such that
\begin{equation}\label{mass}
\int \eta_R(x)\left|u(T,x)\right|^2dx\lesssim \epsilon^2.
\end{equation}

Given the relation  (obtained by multiplying \eqref{INLS} by $\eta_R\bar{u}$ , taking the imaginary part and integrating by parts, see Tao \cite[Section 4]{Tao_Scat} for details)

$$
\partial_t\int \eta_R|u|^2\, dx = 2\Im\int\nabla\eta_R \cdot \nabla u \bar{u},
$$
we have, from \eqref{E}, for all times,
$$
\left| \partial_t \int \eta_R(x)|u(t,x)|^2dx\right| \lesssim \frac{1}{R},
$$

so that,  by \eqref{mass}, for $t \in I_1$,
\begin{equation}
    \int \eta_R(x)\left|u(t,x)\right|^2dx\lesssim \epsilon^2+\frac{\epsilon^{-\alpha}}{R}.
\end{equation}

If $R > \epsilon^{-(\alpha+2)}$, then we have $\left\| \eta_Ru\right\|_{L^\infty_{I_1}L^2_x} \lesssim 
\epsilon
$.


Define $(\hat{q},\hat{r}) \in \mathcal{A}_{s_c}$ as 

\begin{align}
    \hat{q} = \frac{4(p-1)(p+1-\theta)}{(p-1)[N(p-1)+2b]-\theta[N(p-1)-4+2b]},\,\,
    \hat{r} = \frac{N(p-1)(p+1-\theta)}{(p-1)(N-b)-\theta(2-b)}.
\end{align}
We have, by H\"older and Sobolev, for $t \in I_1$,

\begin{equation}\label{recent_past_ball1}
    ||\, \eta_R |x|^{-b}|u|^{p-1}u(t) ||_{L_x^{\hat{r}'}} \lesssim  \|u(t)\|^\theta_{H^1_x} \|u(t)\|^{p-1-\theta}_{L_x^{\hat{r}}}\|\eta_R u(t)\|_{L_x^{\hat{r}}} \lesssim \|\eta_R u(t)\|_{L_x^{\hat{r}}}.
\end{equation}
Now, letting $\hat{\theta}$ be the solution of $\frac{1}{\hat{r}} = \frac{\hat{\theta}}{2}+\frac{1-\hat{\theta}}{p^*}$, we have,
\begin{equation}\label{recent_past_ball2}
    \|\eta_R u(t)\|_{L_x^{\hat{r}}} \leq \|u(t)\|^{1-\hat{\theta}}_{L^{p^*}_x}\|\eta_R u(t)\|^{\hat{\theta}}_{L^2_x} \lesssim \epsilon^{\hat{\theta}},
\end{equation}
uniformly on time in $I_1$. We now exploit the decay of the nonlinearity, instead of assuming radiality\footnote{This is one of the crucial estimates which allow us to drop the radiality assumption.}, to estimate, by H\"older and Sobolev, for $R>0$ large enough (depending on $\epsilon$) and $t \in I_1$,
\begin{align}\label{recent_past_out}
    ||\, (1-\eta_R) |x|^{-b}|u|^{p-1}u(t) ||_{L_x^{\hat{r}'}} &\leq ||\,|x|^{-b}|u|^{p-1}u(t) ||_{L_{\{|x|>R/2\}}^{\hat{r}'}} \\&\leq \|\, |x|^{-b}\|_{L_{\{|x|>R/2\}}^{r_1}} \|u(t)\|^\theta_{L^{\theta r_2}_x} \|u(t)\|^{p-\theta}_{L_x^{\hat{r}}}\\
    & \lesssim \frac{1}{R^{br_1-N}} \|u(t)\|^p_{H^1_x} \lesssim \epsilon^{\hat{\theta}},
\end{align}
where $r_1$ and $r_2$ are such that $br_1 > N$, $\theta r_2 \in (2,N(p-1)/(2-b))$ and
\begin{equation}
    \frac{1}{\hat{r}'} = \frac{1}{r_1} + \frac{1}{r_2} + \frac{p-\theta}{\hat{r}}.
\end{equation}

%
%

Using the local-in-time Strichartz estimate \eqref{lit}, together with estimates \eqref{recent_past_ball1}, \eqref{recent_past_ball2} and \eqref{recent_past_out}, we bound

\begin{align}
\left\| \int_{I_1} e^{i(t-s)\Delta}|x|^{-b}|u|^{p-1}u(s) \,ds\right\|_{S(\dot{H}^{s_c},[T,+\infty))} &\leq ||\, |x|^{-b}|u|^{p-1}u ||_{S'(\dot{H}^{-s_c},I_1)}\\
&\hspace{-5cm}\leq ||\, \eta_R |x|^{-b}|u|^{p-1}u ||_{L^{\hat{q}'}_{I_1}L_x^{\hat{r}'}} + ||\, (1-\eta_R) |x|^{-b}|u|^{p-1}u ||_{L^{\hat{q}'}_{I_1}L_x^{\hat{r}'}}\\
&\hspace{-5cm}\lesssim|I_1|^{1/\hat{q}'}\epsilon^{\hat{\theta}} = \epsilon^{\hat{\theta}-\alpha/\hat{q}'}= \epsilon^{\hat{\theta}/2},
\end{align}
where we chose $\alpha := \hat{q}'\hat{\theta}/{2}$.
 
%
%

\textbf{Step 2. Estimate on distant past.}

The estimate for the distant past is the same as in \cite{Campos_New_2019}, as radiality does not play a role in this part of the estimate. We provide the argument here for completeness. Let  $(q, r) \in \mathcal{A}_{s_c}$. 
Define, for small $\delta>0$,
\begin{equation}
	\frac{1}{c} = \left(\frac{1}{1-s_c}\right)\left[\frac{1}{q}-\delta s_c\right]
\end{equation}
and
\begin{equation}
    \frac{1}{d} = \left(\frac{1}{1-s_c}\right)\left[\frac{1}{r}-s_c\left(\frac{N-2-4\delta}{2N}\right)\right].
\end{equation}

We see that $(c,d) \in \mathcal{A}_0$ (see \cite[Section 3]{Campos_New_2019}). By interpolation,

$$\left\|F_2\right\|_{L_{[T,+\infty)}^{q}L_x^{r}}  
\leq 
\left\|F_2\right\|^{1-s_c}_{L_{[T,+\infty)}^{c}L_x^{d}}
\left\|F_2\right\|^{s_c}_{L_{[T,+\infty)}^{\frac{1}{\delta}}L_x^{\frac{2N}{N-2-4\delta}}}.
$$ Using Duhamel's principle, write
$$
F_2 = e^{it\Delta}\left[e^{i(-T+\epsilon^{-\alpha})\Delta}u(T-\epsilon^{-\alpha})-u(0)\right].
$$

Thus, by the Strichartz estimate \eqref{S1},
\begin{align}
\left\| F_2\right\|_{L_{[T,+\infty)}^{q}L_x^{r}}
&\leq\left\| e^{it\Delta}\left[e^{i(-T+\epsilon^{-\alpha})\Delta}u(T-\epsilon^{-\alpha})-u(0)\right]\right\|^{1-s_c}_{L_{[T,+\infty)}^{c}L_x^{d}}
\left\|F_2\right\|^{s_c}_{L_{[T,+\infty)}^{\frac{1}{\delta}}L_x^{\frac{2N}{N-2-4\delta}}}\\
&\lesssim\left(\left\|u\right\|_{L^\infty_t L^2_x}\right)^{1-s_c} \left\|F_2\right\|^{s_c}_{L_{[T,+\infty)}^{\frac{1}{\delta}}L_x^{\frac{2N}{N-2-4\delta}}}
\lesssim \epsilon^{\alpha \delta s_c},
\end{align}

since, by \eqref{decay} and \eqref{L1},

\begin{align}
\left\|F_2\right\|_{L_{[T,+\infty)}^{\frac{1}{\delta}}L_x^{\frac{2N}{N-2-4\delta}}}
&\lesssim \left\|\int_{I_2} |\cdot-s|^{-(1+2\delta)}\left\| |x|^{-b} |u|^{p-1}u(s) \right\|_{L^{\frac{2N}{N+2+4\delta}}_x}\, ds\right\|_{L_{[T,+\infty)}^{\frac{1}{\delta}}}\\
&\lesssim \|u\|_{L_t^\infty H_x^1}^{p}\left\|\left(\cdot -T+\epsilon^{-\alpha}\right)^{-2\delta}\right\|_{L_{[T,+\infty)}^{\frac{1}{\delta}}}\\
&\lesssim \epsilon^{\alpha\delta}.
\end{align}

Therefore, defining $\gamma := \min\{ \hat{\theta}/2, \alpha \delta s_c\}$ and recalling that

$$
e^{i(t-T)\Delta}u(T) = e^{it\Delta}u_0 + iF_1 + iF_2,
$$
we have


$$
\left\|e^{i(\cdot-T)\Delta} u(T)\right\|_{S\left(\dot{H}^{s_c}, [T,+\infty) \right)}  \lesssim\epsilon^\gamma.
$$


Hence, Lemma \ref{linevo} is proved.
\end{proof}
\begin{proof}[Proof of Theorem \ref{scattering_criterion}]

Choose $\epsilon$ is small enough so that, by Lemma \ref{linevo}, $$\left\|e^{i(\cdot)\Delta}u(T)\right\|_{S\left(\dot{H}^{s_c}, [0, +\infty)\right)}  = \left\|e^{i(\cdot-T)\Delta} u(T)\right\|_{S\left(\dot{H}^{s_c}, [T, +\infty)\right)} \leq c\epsilon^\gamma\leq \delta_{sd},
 $$
 
 where $\delta_{sd}$ is given in Lemma \ref{small_data}. Thus, by small data theory, $u$ scatters forward in time in $H^1$, as desired.
\end{proof}

\section{Proof of scattering}
We now turn to Theorem \ref{teo1}. The main idea behind the proof is to combine the decay of the nonlinearity (instead of exploiting some form of radial decay) with a truncated Virial identity. By choosing a suitable weight, and employing coercivity on large balls around the origin, one can control a time-averaged weighed $L^p$ norm on these balls. Averaging is necessary due to the lack of uniform estimates in time, since we are not employing concentration-compactness as in Holmer-Roudenko \cites{HR_Scat,DHR_Scat}.

We start with the following ``trapping'' lemmas, which can be found in \cite{Campos_New_2019} and in Farah and Guzm\'an \cite{FG_scat}.
\begin{lem}[Energy trapping]\label{coercivity}
Let $N \geq 1$ and $0 < s_c < 1$. If  $$M[u_0]^\frac{1-s_c}{s_c}E[u_0] < (1-\delta)M[u_0]^\frac{1-s_c}{s_c}E[u_0]$$ for some $\delta > 0$ and $$\|u_0\|_{L^2}^\frac{1-s_c}{s_c}\|\nabla u_0\|_{L^2} \leq \|Q\|_{L^2}^\frac{1-s_c}{s_c}\|\nabla Q\|_{L^2},$$ then there exists $\delta' = \delta'(\delta) > 0$ such that
$$
\|u_0\|_{L^2}^\frac{1-s_c}{s_c}\|\nabla u_0\|_{L^2} < (1-\delta') \|Q\|_{L^2}^\frac{1-s_c}{s_c}\|\nabla Q\|_{L^2}.
$$ for all $t \in I$, where $I \subset \mathbb{R}$ is the maximal interval of existence of the solution $u(t)$ to \eqref{INLS}. Moreover, $I = \mathbb{R}$ and $u$ is uniformly bounded in $H^1$.
\end{lem}

\begin{lem}\label{coercivity2}
Suppose, for $f \in H^1(\mathbb{R}^N)$, $N \geq 1$, that $$\|f\|_{L^2}^\frac{1-s_c}{s_c}\|\nabla f\|_{L^2} < (1-\delta) \|Q\|_{L^2}^\frac{1-s_c}{s_c}\|\nabla Q\|_{L^2}.$$ Then there exists $\delta' = \delta'(\delta) > 0$  so that 
$$
 \int|\nabla f|^2 + \left(\frac{N-b}{p+1}-\frac{N}{2}\right)\int|x|^{-b}|f|^{p+1} \geq \delta'\int|x|^{-b}|f|^{p+1}.
$$
\end{lem}

From now on, we consider $u$ to be a solution to \eqref{INLS} satisfying the conditions \eqref{ME<1} and \eqref{MK<1}.
In particular, by Lemma \ref{coercivity}, $u$ is global and uniformly bounded in $H^1$. Moreover, there exists $\delta > 0$ such that 
\begin{equation}\label{condition_u}
\sup_{t \in \mathbb{R}}\|u_0\|_{L^2}^\frac{1-s_c}{s_c}\|\nabla u(t)\|_{L^2} <(1-2\delta) \| Q\|_{L^2}^\frac{1-s_c}{s_c}\|\nabla Q\|_{L^2}
\end{equation}
In the spirit of Dodson and Murphy \cite{MD_New}, local coercivity was proved in \cite{Campos_New_2019}. They proved:
\begin{lem}\label{lem_comut} For $N \geq 1$, let $\phi$ be a smooth cutoff to the set $\{|x|\leq\frac{1}{2}\}$ and define $\phi_R(x) = \phi\left(\frac{x}{R}\right)$. If $f \in H^1(\mathbb{R}^N)$ , then
\begin{equation}\label{comut}
\int|\nabla(\phi_Rf)|^2=\int\phi_R^2|\nabla f|^2-\int \phi_R\Delta(\phi_R)|f|^2.
\end{equation}
In particular, 
\begin{equation}\label{comut2}
\left|\int|\nabla(\phi_Rf)|^2 - \int\phi_R^2|\nabla f|^2\right|\leq\frac{c}{R^2}\|f\|^2_{L^2}.
\end{equation}
\end{lem}

\begin{lem}[Local coercivity]\label{coercivity_balls} For $N \geq 1$, let $u$ be a globally defined $H^1(\mathbb{R}^N)$-solution to \eqref{INLS} satisfying \eqref{condition_u}. There exists $ \bar{R} =\  \bar{R}(\delta, M[u_0],Q, s_c)>0$ such that, for any $R \geq \bar{R}$,
\begin{equation}
\sup_{t \in \mathbb{R}}\|\phi_Ru(t)\|^\frac{1-s_c}{s_c}_{L^2}\|\nabla(\phi_R u(t))\|_{L^2} \leq (1-\delta)\|Q\|^\frac{1-s_c}{s_c}_{L^2}\|\nabla Q\|_{L^2}.
\end{equation}
In particular, by Lemma \ref{coercivity2}, there exists $\delta' = \delta'(\delta) >0$ such that 
\begin{equation}\label{local_coercivity}
\int|\nabla ( \phi_R u(t))|^2 + \left(\frac{N-b}{p+1}-\frac{N}{2}\right)\int|x|^{-b}|\phi_R u(t)|^{p+1} \geq \delta'\int|x|^{-b}|\phi_Ru(t)|^{p+1}.
\end{equation}
\end{lem}
We exploit the coercivity given by the previous lemma by making use of the Virial identity (see Dodson and Murphy \cite[Lemma 3.3]{MD_New}, Farah and Guzm\'an \cite[Proposition 7.2]{FG_scat})

\begin{lem}[Virial identity]\label{Morawetz}
Let $a: \mathbb{R}^N \rightarrow \mathbb{R}$ be a real-valued weight. If $|\nabla a| \in L^\infty$, define
$$
Z(t) = 2 \Im\int \bar{u} \nabla u \cdot \nabla a \, dx.
$$
Then, if $u$ is a solution to \eqref{INLS}, we have the following identity 
\begin{align}
\frac{d}{dt}Z(t)& = \left(\frac{4}{p+1}-2\right)\int|x|^{-b}|u|^{p+1} \Delta a  -\frac{4b}{p+1}\int|x|^{-b-2}|u|^{p+1}x\cdot\nabla a \\
&\quad
-\int|u|^2\Delta\Delta a + 4\Re\sum_{i,j}\int a_{ij}\bar{u}_i u_j.
\end{align}
\end{lem}
We now have all the basic tools needed to prove scattering. Let $R \gg 1$   to be determined below. We take $a$ to be a smooth radial function satisfying
$$
a(x) = \begin{cases}|x|^2 & |x|\leq \frac{R}{2}, \\
2R|x|& |x| > R. \\
\end{cases}
$$
In the intermediate region $\frac{R}{2} < |x| \leq R$, we impose that
$$
\partial_ra \geq 0, \,\,\, \partial_r^2a \geq 0, \,\,\,|\partial^\alpha a(x)| \lesssim_{\alpha}R |x|^{-|\alpha|+1} \,\,\, \text{for}\,\,\,|\alpha| \geq 1.
$$
Here, $\partial_r$ denotes the radial derivative, i.e., $\partial_r a = \nabla a \cdot \frac{x}{|x|}$. Note that for $|x| \leq \frac{R}{2}$, we have
$$
a_{ij} = 2 \delta_{ij}, \,\,\, \Delta a = 2N, \,\,\, \Delta \Delta a = 0,
$$
while, for $|x| > R$, we have
$$
a_{ij} = \frac{2R}{|x|}\left[\delta_{ij} - \frac{x_i}{|x|}\frac{x_j}{|x|}\right], \,\,\, \Delta a = \frac{2(N-1)R}{|x|}, \,\,\, |\Delta \Delta a(x)| \lesssim \frac{R}{|x|^3}.
$$

\begin{proof}[Proof of Proposition \ref{virial}]
We follow mostly \cite{Campos_New_2019}, but highlighting the differences (extra terms appearing due to non-radiality and weaker decay) throughout the proof.  Choose $R\geq \bar{R}(\delta, M[u_0], Q, s_c)$ as in Lemma \ref{coercivity_balls}. We define the weight $a$ as above and define $Z(t)$ as in Lemma \ref{Morawetz}. Using Cauchy-Schwarz inequality, and the definition of $Z(t)$, we have
\begin{equation}\label{Mbound}
\sup_{t \in \mathbb{R}} |Z(t)| \lesssim R.
\end{equation}
As in \cite{Campos_New_2019}, we compute
\begin{align}
\frac{d}{dt}Z(t) &=8\left[ \int_{|x|\leq\frac{R}{2}}|\nabla u|^2 + \left(\frac{N-b}{p+1}-\frac{N}{2}\right)\int_{|x|\leq \frac{R}{2}}|x|^{-b}|u|^{p+1}\right]\\
&\quad+\int_{|x| > \frac{R}{2}}\left[\left(\frac{4}{p+1}-2\right)(N-1)\Delta a  -\frac{4b}{p+1}\frac{x\cdot\nabla a }{|x|^2} \right]|x|^{-b}|u|^{p+1}\\
&\quad+4\int_{ |x| > \frac{R}{2}}\partial_r^2a|\partial_r u|^2
+4\int_{ |x| > \frac{R}{2}}\frac{\partial_r a}{|x|}|\slashed{\nabla} u|^2 -\int_{ |x| > \frac{R}{2}}|u|^2 \Delta\Delta a,\\
\end{align}
where we denote 
the angular derivative as $\slashed{\nabla} u = \nabla u - \frac{x\cdot \nabla u}{|x|^2}x$. Note that $\slashed{\nabla} u$ is not necessarily zero, since we are not assuming radiality. Nevertheless, the first two terms in the last line can be dropped, by non-negativity. 

As for the second line, one can bound $\|u\|_{L^{p+1}_x}$ by $E$, using Sobolev, so this term gives us only a decay of $O(1/R^b)$. It is, of course, a weaker decay than that one in \cite{Campos_New_2019} (which used Strauss), but in dimensions $N\geq 3$, it is enough to close the argument . Therefore,

\begin{align}\label{MP}\frac{d}{dt}Z(t)&\geq  8\left[ \int_{|x|\leq\frac{R}{2}}|\nabla u|^2 + \left(\frac{N-b}{p+1}-\frac{N}{2}\right)\int_{|x|\leq \frac{R}{2}}|x|^{-b}|u|^{p+1}\right]\\
&\quad-\frac{cE^\frac{p+1}{2}}{R^b}-\frac{c}{R^2}M[u_0].
\end{align}
Define $\phi^A$ ,  $A>0$, as a smooth cutoff to the set $\{|x| \leq \frac{1}{2}\}$ that vanishes outside the set $\{|x| \leq \frac{1}{2}+\frac{1}{A}\}$, and define $\phi_R^A(x) =\ \phi^A\left(\frac{x}{R}\right)$. In order to use Lemma \ref{coercivity_balls}, we introduce some smoothing in the first term of the last inequality (at the expense of an acceptable error, which decays with a power of R).

\begin{align}\label{MP2}
 \int_{|x|\leq\frac{R}{2}}&|\nabla u|^2 + \left(\frac{N-b}{p+1}-\frac{N}{2}\right)\int_{|x|\leq\frac{R}{2}}|x|^{-b}|u|^{p+1}=\\
= &\left[ \int(\phi^A_R)^2|\nabla u|^2 + \left(\frac{N-b}{p+1}-\frac{N}{2}\right)\int(\phi^A_R)^2|x|^{-b}|u|^{p+1}\right]\\
-&\underbrace{\left[ \int_{\frac{R}{2}<|x|\leq\frac{R}{2}+\frac{R}{A}}(\phi^A_R)^2|\nabla u|^2 + \left(\frac{N-b}{p+1}-\frac{N}{2}\right)\int_{\frac{R}{2}<|x|\leq\frac{R}{2}+\frac{R}{A}}(\phi^A_R)^2|x|^{-b}|u|^{p+1}\right]}_{I_A}\\
=&\left[ \int|\phi^A_R\nabla u|^2 + \left(\frac{N-b}{p+1}-\frac{N}{2}\right)\int|x|^{-b}|\phi^A_Ru|^{p+1}\right]\\
-&I_A-\underbrace{\left(\frac{N}{2}-\frac{N-b}{p+1}\right)\int\left((\phi^A_R)^{p+1}-(\phi^A_R)^{2}\right)|x|^{-b}|u|^{p+1}}_{II_A}.
\end{align}
Using Lemma \ref{lem_comut}, we can write  

\begin{align}\label{MP3}
\int&|\phi^A_R\nabla u|^2 + \left(\frac{N-b}{p+1}-\frac{N}{2}\right)\int|x|^{-b}|\phi^A_Ru|^{p+1}\geq\\
&\int|\nabla (\phi^A_Ru)|^2 + \left(\frac{N-b}{p+1}-\frac{N}{2}\right)\int|x|^{-b}|\phi^A_Ru|^{p+1}-\frac{c}{R^2}M[u_0].
\end{align}

The inequalities \eqref{MP}, \eqref{MP2} and \eqref{MP3} can be rewritten as 
\begin{align}\label{MP4}
\frac{d}{dt}Z(t) &\geq  8 \left[ \int|\nabla (\phi^A_Ru)|^2 + \left(\frac{N-b}{p+1}-\frac{N}{2}\right)\int|x|^{-b}|\phi^A_Ru|^{p+1}\right]\\
&\quad-\frac{cE^\frac{p+1}{2}}{R^b}-\frac{c}{R^2}M[u_0]-8I_A - 8II_A.
\end{align}

By 
Lemma \ref{coercivity_balls}, and recalling that $0 < b < 2$,
we can write \eqref{MP4} as 
\begin{align}
\int|x|^{-b}|\phi^A_Ru(t)|^{p+1} \lesssim \frac{d}{dt}Z(t)+&\frac{1}{R^b}+8I_A +8II_A.
\end{align}

We can now make $A \to +\infty$ to obtain $I_A + II_A \to 0$ by dominated convergence. Hence,

\begin{equation}\label{MP6}
\int_{|x|\leq\frac{R}{2}}|x|^{-b}|u(t)|^{p+1} \lesssim \frac{d}{dt}Z(t)+\frac{1}{R^b}.
\end{equation}

We finish the proof integrating over time, and using \eqref{Mbound}. We have
\begin{align}
\frac{1}{T}\int_0^T \int_{|x|\leq\frac{R}{2}}|x|^{-b}|u(t)|^{p+1} &\lesssim\frac{1}{T}\sup_{t\in[0,T]} |Z(t)| + \frac{1}{R^b}\\
&\lesssim \frac{R}{T} + \frac{1}{R^b}.\end{align}
\end{proof}
We are now able to prove some \textit{energy evacuation}. Note that, unlike in \cite{Campos_New_2019}, and inspired by \cite{Jason2021} we keep the factor $|x|^{-b}$ in the integral, since otherwise it would jeopardize the decay.
\begin{pro}[Energy evacuation]\label{energy_evacuation} Under the hypotheses of Proposition \ref{virial}, there exist a sequence of times $t_n \to +\infty$ and a sequence of radii $R_n \to +\infty$ such that
\begin{equation}\label{eq_energy_evac}
\lim_{n \to +\infty}\int_{|x|\leq R_n}|x|^{-b}|u(t_n)|^{p+1} = 0
\end{equation}
\end{pro}
\begin{proof}
Using Proposition \ref{virial}, choose $T_n \to +\infty$ and $R_n = T_n^{\frac{1}{1+b}}$, so that 
\begin{equation}
\frac{1}{T_n}\int_0^{T_n} \int_{|x|\leq R_n}|x|^{-b}|u(t)|^{p+1}\lesssim \frac{1}{T_n^{\frac{b}{1+b}}}\to 0 \text{ as }n \to +
\infty.
\end{equation}
Therefore, by the Mean Value Theorem, there is a sequence $t_n \to +\infty$ such that \eqref{eq_energy_evac} holds. The proof is complete.
\end{proof}
Using Proposition \ref{energy_evacuation}, we can prove
Theorem \ref{teo1}. We prove only the case $t \to +\infty$, as the case $t \to -\infty$ is entirely analogous.
\begin{proof}[Proof of Theorem \ref{teo1}]Take $t_n \to +\infty$ and $R_n \to +\infty $ as in Proposition \ref{energy_evacuation}. Fix $\epsilon > 0$ and $R>0$ as in Theorem \ref{scattering_criterion}. Choosing $n$ large enough, such that $R_n \geq R$, Hölder's inequality yields
\begin{equation}
\int_{|x|\leq R} |u(x,t_n)|^2 \lesssim R^\frac{2b+N(p-1)}{p+1}\left(\int_{|x|\leq R_n}|x|^{-b}|u(x,t_n)|^{p+1}\right)^\frac{2}{p+1} \to 0 \text{ as } n \to +\infty.
\end{equation}
Therefore, by Theorem \ref{scattering_criterion}, $u$ scatters forward in time.
\end{proof}




\end{document}